\documentclass{amsart}
\usepackage{amsmath,amsthm,amssymb,cases,color,cite}
\usepackage[all]{xy}

\theoremstyle{thmit} 
\newtheorem{theorem}{Theorem}[section]
\newtheorem{lemma}[theorem]{Lemma}
\newtheorem{corollary}[theorem]{Corollary}
\newtheorem{proposition}[theorem]{Proposition}

\theoremstyle{thmrm} 
\newtheorem{example}{Example}
\newtheorem*{remark}{Remark}
\newtheorem*{oldproof}{Proof}
\renewenvironment{proof}[1][{}]{\begin{oldproof}[#1]}{\qed\end{oldproof}}
\newtheorem{definition}[theorem]{Definition}

\numberwithin{equation}{section}

\title{Travel groupoids on complete multipartite graphs}
\author{Diogo Kendy Matsumoto}

\address{Center for Arts and Sciences,\\ TEIKYO University of Science,\\ 2525 Yatsusawa, Uenohara, Yamanashi 409-0193. Japan.}
\email{diogo-swm@ntu.ac.jp}


\begin{document}
\maketitle

\begin{abstract}
A travel groupoid is an algebraic system satisfying two suitable conditions, which has a relation to graphs.
In this article, we characterize travel groupoids on finite complete multipartite graphs, 
and we give the numbers of travel groupoids on the complete multipartite graphs.
\end{abstract}

\section{Introduction}\label{sec1}
The conception of travel groupoids was introduced by Nebesk\'y in 2006 \cite{Neb2006}, 
as a generalization of his study about an algebraic characterization of geodetic graphs \cite{Neb98,Neb02} and trees \cite{Neb00}.
A geodetic graph signifies a connected graph having exactly one shortest path between any two vertices.

In this article, graphs are finite, undirected and have no multiple edges or loops.
For a graph $G$, $V(G)$ denotes the vertex set of $G$ and $E(G)$ denotes the edge set of $G$.

\begin{definition}
A groupoid is a pair $(V,*)$ of a nonempty set $V$ and a binary operation $*:V\times V\rightarrow V$.
\end{definition}

\begin{definition}\label{travel groupoid}{\bf (Travel groupoid)}
A groupoid $(V,*)$ is called a \textit{travel groupoid} if it satisfies the following two conditions:\\
{\rm (t1)}\ $(u*v)*u=u$ (for all $u, v\in V$);\\
{\rm (t2)}\ if $(u*v)*v=u$, then $u=v$ (for all $u, v\in V$).
\end{definition}

Let $(V, *)$ be a travel groupoid.
{\it The associated graph} of $(V,*)$ is a graph $G_{V}$ defined by a vertex set $V(G_V)=V$ and an edge set
\[
 E(G_V)=\{ \{ u,v\} | u,v\in V\ \mbox{and}\  u\neq u*v=v \}.
\]
For a graph $G$, we say that $(V, *)$ is {\it on} $G$ or that $G$ {\it has} $(V, *)$ if $G = G_{V}$.
Note that, every travel groupoid is on exactly one graph, 
but many different travel groupoids may be on the same graph.

In \cite{Neb2006}, Nebesk\'y defined some classes of travel groupoids and proposed three questions.
In recent years, some of these questions are solved in \cite{CPS2,MM16}.
Furthermore, a characterization of an important class called a non-confusing travel groupoid are given in \cite{CPS1} by using spanning trees.
%
For more studies, we refer to \cite{CP19,CPS3,MM19}.

The aim of this article is to study about travel groupoids on a complete multipartite graph.
A complete multipartite graph $G$ with $l$ parts is a graph whose vertices $V(G)$ are partitioned into different independent sets $V_1,\cdots ,V_l$, and having an edge between every pair of vertices from different independent sets.
When $V_i$ has $n_i$ vertices ($1\leq i\leq l$), we denote the complete multipartite graph by $K_{n_{1},\cdots ,n_{l} }$.
Especially, a complete multipartite graph with $2$ parts is called a complete bipartite graph.

\begin{remark}
The complete multipartite graph includes some special classes.
For examples, 
the edgeless graph is a complete multipartite graph with one independent set, 
and the complete graph is a complete multipartite graph whose independent sets consists of one vertex.
\end{remark}

The organization of this article is as follows.
In section 2, we review some definitions and results about travel groupoids.
In section 3, we give a characterization of travel groupoids on a complete multipartite graph. Furthermore, we characterize some special classes of complete multipartite graphs.
In section 4, we give the numbers of travel groupoids on complete multipartite graphs.

\section{Preliminaries}\label{sec2}

\subsection{Basic results}

\begin{proposition}[\cite{Neb2006}]\label{tg-pro}
Let $(V,*)$ be a travel groupoid. Then the following relations hold.
\begin{enumerate}
 \item $u*u=u$ (for all $u\in V$),
 \item $u*v=v$ if and only if $v*u=u$ (for all $u,v\in V$),
 \item $u*v=u$ if and only if $u=v$ (for all $u,v\in V$),
 \item $u*(u*v)=u*v$ (for all $u,v\in V$).
\end{enumerate} 
\end{proposition}

\begin{proposition}[\cite{Neb2006}]\label{tg-pro2}
Let $(V,*)$ be a travel groupoid on a graph $G$ and let $u,v\in V(u\neq v)$. 
Then $u$ and $u*v$ are adjacent vertices of $G$. 
\end{proposition}

Let $(V,*)$ be a travel groupoid.
Then, for  $u, v\in V$, we define the following notation
\[
u*^{0}v=u
\] 
and 
\[
u*^{i+1}v=(u*^{i}v)*v
\] 
for all integer $i\geq 0$.

\begin{proposition}[\cite{Neb2006}]
Let $(V,*)$ be a travel groupoid on a graph $G$. 
For $u,v\in V$, and integer $k\geq 1$, assume that $u*^{k-1}v\neq v$. 
Then the sequence of vertices
\begin{equation}\label{walk-01}
 u=u*^{0}v, u*v,\cdots ,u*^{k-1}v , u*^{k}v
\end{equation}
is a walk in $G$.
Moreover, if $u*^{k}v=v$, then the sequence (\ref{walk-01}) is an $u$-$v$ path in $G$.
\end{proposition}

\begin{theorem}[\cite{Neb2006}]
Let $G$ be a finite graph. 
Then $G$ has a travel groupoid if and only if $G$ is connected 
or $G$ is disconnected and no component of $G$ is a tree.
\end{theorem}

\begin{remark}\label{rem-exists}
This theorem implies that there is no travel groupoid on an edgeless graph $G$ except the case $|V(G)|=1$.
\end{remark}

\subsection{Non-confusing travel groupoids}

\begin{definition}
Let $(V,*)$ be a travel groupoid and $u,v\in V (u\neq v)$.
\begin{enumerate}
\item {\bf (Confusing pair)} An ordered pair $(u,v)$ is called a \textit{confusing pair} in $(V,*)$, if there exists $i\geq 3$ such that $u*^{i}v=u$.
\item {\bf (Non-confusing)}  If $(V,*)$ has no confusing pair, then we call $(V,*)$ a \textit{non-confusing} travel groupoid.
\end{enumerate}
\end{definition}

\begin{proposition}{\rm (\cite{Neb2006})}\label{NC-prop1}
Let $(V,*)$ be a finite travel groupoid on a graph $G$. 
Then $(V,*)$ is non-confusing if and only if the following statement holds for all $u,v\in V (u\neq v)$:\\
\quad there exists $k\geq 1$ such that the sequence of vertices
\[
 u=u*^{0}v, u*v,\cdots ,u*^{k}v=v
\]
\quad is an $u$-$v$ path in $G$.
\end{proposition}

\begin{definition}\label{v-tree}{\bf ($v$-spanning tree)}
Let $G$ be a graph and $v \in V(G)$.
We say that a spanning tree $T$ of $G$ is a $v$-spanning tree if $T$ contains all incident edges of the vertex $v$.
We write $S_{G}(v)$ for the set of all $v$-trees in $G$.
\end{definition}
\begin{theorem}{\rm (\cite{CPS1})}\label{nc-tree}
Let $G$ be a finite connected graph.
Then, there exists a one-to-one correspondence between the set $\prod_{v\in V(G)}S_{G}(v)$ 
and the set of all non-confusing travel groupoids on $G$.
\end{theorem}

Since this theorem plays an important role in the study of travel groupoids, 
we refer to the correspondence explicitly here.

Let $(V,*)$ be a non-confusing travel groupoid on a finite graph $G$.
Then, for each $v\in V$, the graph $T_{v}$ defined by
\[
 V(T_{v}):=V
\]
and
\[
 E(T_{v}):=\{ \{ u,u*v\} | u\in V, u\neq v \}
\]
is a $v$-spanning tree of $G$. 
Hence, $\{ T_{v}\}_{v\in V}\in \prod_{v\in V(G)}S_{G}(v)$.

For the converse correspondence, 
let $\{ T_{v}\}_{v\in V}\in \prod_{v\in V(G)}S_{G}(v)$ 
and let $A_{T_{v}}(u,v)$ be the vertex adjacent to $u$ which is on the unique $u$-$v$ path in the $v$-spanning tree $T_{v}$.
Define a binary operation on $V=V(G)$ as follows:
\[
u*v=
\begin{cases}
A_{T_{v}}(u,v)  &  u\neq v \\
u               &  u=v
\end{cases}
\]
Then $(V,*)$ is a non-confusing travel groupoid on $G$.\\


\begin{corollary}{\rm (\cite{CPS1})}\label{nc-number}
Let $G$ be a finite connected graph. 
Then, the number of non-confusing travel groupoids on $G$ is equal to $\prod_{v\in V (G)} \left| S_{G}(v) \right| $.
\end{corollary}

\subsection{Simple travel groupoids}

\begin{definition}
{\bf (Simple)}  A travel groupoid $(V,*)$ is called simple if it satisfies the following condition:\\
{\rm (t3)}\ if $v*u\neq u$, then $u*(v*u)=u*v$ (for all $u,v\in V$).
\end{definition}

\begin{theorem}{\rm (\cite{Neb2006})}\label{NC-ex}
For every finite connected graph $G$ there exists a simple non-confusing travel groupoid on $G$.
\end{theorem}

\begin{theorem}{\rm (\cite{MM19})}\label{simple-tree}
Let $G$ be a finite connected graph and define the set
\[
 \mathfrak{S}_G=\left\{ \{ T_v \}\in \prod_{v\in V}S_G(v) |  \exists u-v\mbox{ path on } T_u\cap T_v \ (\forall u,v\in V) \right\}
\] 
Then there exists a one-to-one correspondence between $\mathfrak{S}_G$ and the set of all simple non-confusing travel groupoids on $G$.
\end{theorem}

\subsection{Semi-smooth and smooth travel groupoids}

\begin{definition}
\begin{enumerate}
\item {\bf (Smooth)} A travel groupoid $(V,*)$ is called smooth if it satisfies the following condition:\\
{\rm (t4)}\ if $u*v=u*w$, then $u*(w*v)=u*v$ (for all $u,v,w\in V$).
\item {\bf (Semi-smooth)} A travel groupoid $(V,*)$ is called semi-smooth if it satisfies the following condition:\\
{\rm (t5)}\ if $u*v=u*w$, then $u*(v*w)=u*v$ or $u*((v*w)*w)=u*v$ (for all $u,v,w\in V$).
\end{enumerate}
\end{definition}

\begin{proposition}{\rm (\cite{Neb2006})}\label{semismooth-nonconf}
Every semi-smooth travel groupoid is non-confusing.
\end{proposition}

\begin{theorem}{\rm (\cite{MM16})}\label{smooth-exist}
For any finite connected graph $G$, there exists a smooth travel groupoid on $G$.
\end{theorem}

For some special graphs, the following results are known.

\begin{proposition}{\rm (\cite{Neb2006})}
For any complete bipartite graph $G$, there exists a simple smooth travel groupoid on $G$.
\end{proposition}

\begin{theorem}{\rm (\cite{Neb2006})}
Let $G$ be a geodetic graph of diameter 2,
and let $(V,*)$ be the proper groupoid on $G$.
Then $(V,*)$ is a smooth travel groupoid.
\end{theorem}

\section{Travel groupoids on complete multipartite graphs}
In this section, we give a characterization of a travel groupoid on complete multipartite graphs.
Firstly, we reconsider a complete multipartite graph from the viewpoint of the condition of edges.
\begin{lemma}\label{edge-condition}
A graph $G$ is a complete multipartite graph, if and only if $G$ satisfies the following condition:
\\
if $vw\in E(G)$ then $vu\in E(G)$ or $wu\in E(G)$ (for all pairwise distinct $u,v,w\in V(G)$).
\end{lemma}
\begin{proof}
If $G$ is an edgeless graph or $|V(G)|\leq 2$ the lemma holds clearly.

Let $G$ be a complete multipartite graph and $u,v,w$ are pairwise distinct vertices of $G$.
If $vw\in E(G)$, the vertices $v$ and $w$ are in different independent sets.
Thus, $u$ adjacent to $v$ or $w$.
This means $uv \in E(G)$ or $uw \in E(G)$.

Conversely, let $G$ be a graph satisfying the above condition. 
We prove that $V(G)$ can be partitioned into some independent sets, and any vertices of different independent sets are adjacent.
Define a relation $\sim$ on $V(G)$ by
\[
 u\sim v :\iff u=v \mbox{ or } uv\not\in E(G)
\]
for all $u,v\in V(G)$. 
Clearly, the relation $\sim$ is reflexive and symmetric. 
From the contraposition of the above condition, 
we obtain the transitivity of $\sim$.
Hence, the relation $\sim$ is an equivalence relation on $V(G)$.
Denote the equivalence class of $v\in V(G)$ by $[v]$, 
and let $[v_1],\cdots ,[v_l]$ be all equivalence classes of $V(G)$.
Thus, we obtain a partition
\[
 V(G)=[v_1]\cup \cdots \cup [v_l].
\]
From the definition, obviously, each equivalence class $[v_i]$ is an independent set of $G$ and every pair of vertices from different equivalence classes is adjacent. Therefore $G$ is a complete multipartite graph with $l$ parts.
\end{proof}

By rewriting the above condition from the viewpoint of the travel groupoid, we obtain the following  characterization.

\begin{remark}
From Remark \ref{rem-exists}, 
note that a travel groupoid is on a complete multipartite graph $G$ 
if and only if 
$G$ is not an edgeless graph with $|V(G)|>1$.
\end{remark}

\begin{theorem}\label{tra-on-compmulti}
Let $(V,*)$ be a travel groupoid.
Then the following statements are equivalent:
\begin{enumerate}
\item $(V,*)$ is on a complete multipartite graph;
\item $(V,*)$ satisfies the following condition;
\\
 {\rm (tcm)} if $v*w=w$, then  $v*u=u$ or $w*u=u$ (for all pairwise distinct $u,v,w\in V$).
\end{enumerate}
\end{theorem}
\begin{proof}
By the definition of the associated graph $G_V$, 
the condition (tcm) corresponds to the following condition:
\\
if $vw\in E(G_V)$ then $vu\in E(G_V)$ or $wu\in E(G_V)$ (for all pairwise distinct $u,v,w\in V(G)$).
\\
Thus, the theorem follows from Lemma \ref{edge-condition}.
\end{proof}


The rest of this section, we study a groupoid satisfying the condition (tcm) and some properties of a travel groupoid on a complete multipartite graph. 
\begin{lemma}\label{diam-2}
Let $(V,*)$ be a groupoid satisfying (t1) and (tcm).
Then the following statements are equivalent:
\begin{enumerate}
\item $(V,*)$ satisfies (t2);
\item $(V,*)$ satisfies $u* u=u$ for all $u\in V$;
\item $(V,*)$ satisfies $u*^2 v=v$ for all $u,v\in V$.
\end{enumerate}
\end{lemma}

\begin{proof}
Suppose that $(V,*)$ satisfies (t2), then $(V,*)$ is a travel groupoid and $u*u=u$ (for all $u\in V$) follows from Proposition \ref{tg-pro}.

Let $(V,*)$ satisfy $u*u=u$ for all $u\in V$.
If $u,v\in V$ satisfies $u*v=v$ or $u=v$, obviously we obtain $u*^2v=v$.
Suppose that there exist $u,v\in V (u\neq v)$ satisfying $u*v\neq v$.
Then, by (t1) and (tcm), $(u*v)*u=u$ induces
\[
u*^2 v=(u*v)*v=v.
\]
Thus, we obtain $u*^2 v=v$ for all $u,v\in V$. 

Lastly, suppose that $(V,*)$ satisfies $u*^2 v=v$ for all $u,v\in V$.
Then (t2) follows from this condition clearly.
\end{proof}

\begin{definition}
A groupoid $(V,*)$ is called an idempotent if it satisfies $u*u=u$ for all $u\in V$.
\end{definition}

\begin{remark}
By Proposition \ref{tg-pro}, travel groupoid is an idempotent groupoid.
\end{remark}


Using Lemma \ref{diam-2}, we can prove that a travel groupoid on a complete multipartite graph corresponds to the following groupoid.

\begin{theorem}
Let $(V,*)$ be a groupoid.
Then the following statements are equivalent:
\begin{enumerate}
\item $(V,*)$ is a travel groupoid on a complete multipartite graph;
\item $(V,*)$ is an idempotent groupoid satisfying the conditions (t1) and (tcm).
\end{enumerate}
\end{theorem}

\begin{proof}
This follows from Theorem \ref{tra-on-compmulti} and Lemma \ref{diam-2}.
\end{proof}

Clearly, if a travel groupoid $(V,*)$ satisfies $u*^2 v=v$ for all $u,v\in V$, then the associated graph $G_V$ has diameter less than 2. 
The converse is not always true (see Remark \ref{rem-not-diam2}), but, for a travel groupoid satisfying the condition (tcm), namely for a travel groupoids on complete multipartite graphs, we can obtain the converse.

\begin{corollary}\label{prop-diam2}
A travel groupoid $(V,*)$ on a complete multipartite graph is non-confusing 
and satisfies $u*^2 v=v$ for all $u,v\in V$.
\end{corollary}
\begin{proof}
Since $(V,*)$ satisfies (t1) and (tcm), it follows from Lemma \ref{diam-2}.
\end{proof}

\begin{remark}\label{rem-not-diam2}
Corollary \ref{prop-diam2} is not trivial consequence from the diameter of graph.
For example, the following travel groupoid $(V,*)$ is on a graph $G$ which has diameter 2, but $u_1,u_4\in V$ satisfies 
$u_4*^2 u_1=u_2\neq u_1$.
\[
\begin{tabular}{c|cccc}
$ *$ & $u_{1}$ & $u_{2}$ & $u_{3}$ & $u_{4}$ \\ \hline
$u_{1}$  & $u_{1}$ & $u_{2}$ & $u_{2}$ & $u_{2}$   \\
$u_{2}$  & $u_{1}$ & $u_{2}$ & $u_{3}$ & $u_{4}$   \\
$u_{3}$  & $u_{2}$ & $u_{2}$ & $u_{3}$ & $u_{4}$   \\
$u_{4}$  & $u_{3}$ & $u_{2}$ & $u_{3}$ & $u_{4}$  
\end{tabular}
\quad \quad
G:
\begin{xy}
{(0,0) *{u_2}="a_1"\ar @{-} (0,15) *{u_1}="a_1"},
{(0,0) *{u_2}="a_4"\ar @{-} (15,-15) *{u_4}="a_5"},
{(0,0) *{u_2}="a_4"\ar @{-} (-15,-15) *{u_3}="a_5"},
{(-15,-15) *{u_3}="a_2"\ar @{-} (15,-15) *{u_4}="a_3"},
\end{xy}
\]

\end{remark}

For the important classes, semi-smooth travel groupoids and smooth travel groupoids, we obtain the following results.

\begin{theorem}\label{semi-smooth-tra}
A travel groupoid on a complete multipartite graph is semi-smooth.
\end{theorem}
\begin{proof}
It follows from Corollary \ref{prop-diam2}.
\end{proof}

\begin{theorem}\label{smooth-tra}
A travel groupoid $(V,*)$ on a complete multipartite graph $K_{n_1,\cdots ,n_l}$ $(n_1,\cdots ,n_l \leq 2)$ is smooth. 
\end{theorem}
\begin{proof}
Suppose that $u*v=u*w$ for $u,v,w\in V (v\neq w)$.
If $u=v$ or $u=w$, we obtain $v=w$ by Proposition \ref{tg-pro}. 
This is a contradiction. 
Thus, we can take pairwise different vertices $u,v,w\in V$.
If $u*v\neq v$ and $u*w\neq w$, then we obtain $v*w\neq w$ by the contraposition of (tcm), 
and $u,v,w$ make an independent set. This is a contradiction to $n_i\leq 2$.
Suppose that $u*v\neq v$ and $u*w=w$. Then, using (tcm) again, we obtain $w*v=v$.
Thus, 
\[
 u*(w*v)=u*v.
\]
Therefore, $(V,*)$ is smooth. 
\end{proof}

\begin{example}
The following table gives a travel groupoid on complete bipartite graph $G=K_{2,3}$.
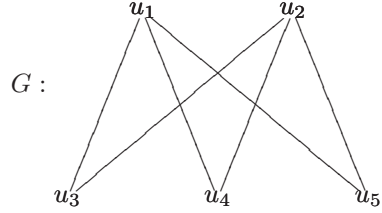
\[
\begin{tabular}{c|ccccc}
$ *$ & $u_{1}$ & $u_{2}$ & $u_{3}$ & $u_{4}$ & $u_{5}$  \\ \hline
$u_{1}$  & $u_{1}$ & $u_{3}$ & $u_{3}$ & $u_{4}$ & $u_{5}$  \\
$u_{2}$  & $u_{5}$ & $u_{2}$ & $u_{3}$ & $u_{4}$ & $u_{5}$  \\
$u_{3}$  & $u_{1}$ & $u_{2}$ & $u_{3}$ & $u_{1}$ & $u_{1}$  \\
$u_{4}$  & $u_{1}$ & $u_{2}$ & $u_{2}$ & $u_{4}$ & $u_{2}$  \\
$u_{5}$  & $u_{1}$ & $u_{2}$ & $u_{2}$ & $u_{2}$ & $u_{5}$  
\end{tabular}
\quad \quad 
G:
\begin{xy}
{(10,10) *{u_1}="a_1"\ar @{-} (0,-15) *{u_3}="a_1"},
{(10,10) *{u_1}="a_1"\ar @{-} (20,-15) *{u_4}="a_1"},
{(10,10) *{u_1}="a_1"\ar @{-} (40,-15) *{u_5}="a_1"},
{(30,10) *{u_2}="a_1"\ar @{-} (0,-15) *{u_3}="a_1"},
{(30,10) *{u_2}="a_1"\ar @{-} (20,-15) *{u_4}="a_1"},
{(30,10) *{u_2}="a_1"\ar @{-} (40,-15) *{u_5}="a_1"},
\end{xy}
\]
This travel groupoid is not smooth.
Indeed, $u_{3}*u_{4}=u_{3}*u_{5}=u_{1}$ but $u_{3}*(u_{4}*u_{5})=u_{2}\neq u_{1}=u_{3}*u_{5}$.
\end{example}

\section{Travel groupoids on complete graphs, on complete bipartite graphs, on star graphs}
In this section, we give characterizations of travel groupoids on a complete graph, on a complete bipartite graph and on a star graph. 
These graphs are special classes of complete multipartite graph.
For the first, we give a characterization of a travel groupoid on a complete graph.
\begin{theorem}\label{asso-tra}
Let $(V,*)$ be a groupoid. 
Then the following statements are equivalent:
\begin{enumerate}
\item $(V,*)$ is a travel groupoid on a complete graph.
\item $(V,*)$ satisfies $u*v=v$ for all $u,v\in V$.
\item $(V,*)$ is a travel groupoid satisfying the associativity
\[
 (u*v)*w=u*(v*w)
\]
for all $u,v,w\in V$. 
\end{enumerate}
\end{theorem}
\begin{proof}
The equivalence of 1 and 2 is follows from Proposition \ref{tg-pro2}.

Let $(V,*)$ be a travel groupoid satisfying the associativity.
Then, by (t1) 
\[
(u*v)*u=u*(v*u)=u.
\]
Thus, by Proposition \ref{tg-pro}, we obtain $v*u=u$ for all $u,v\in V$. 
This means that any two vertices are adjacent.
Conversely, since any two vertices of complete graph are adjacent, 
a travel groupoid on the complete graph satisfies $x*y=y$ for all vertices $x$ and $y$. 
Thus, we obtain
\[
 (u*v)*w=w=u*(v*w).
\]
\end{proof}

By Theorem \ref{asso-tra}, we have the following two results.
Here, a clique of a graph means a set of mutually adjacent vertices, 
and a subgroupoid of $(V,*)$ means a subset closed under the binary operation $*$. 
Note that, a subgroupoid of a travel groupoid is also a travel groupoid.
\begin{corollary}\label{asso-clique}
Let $(V,*)$ be a travel groupoid on a graph $G$.
Then, there exist one to one correspondence between  
an associative subgroupoid of $V$ and a clique of $G$. 
\end{corollary}

\begin{corollary}
A travel groupoid on a complete multipartite graph with $l$ parts has a $l$ elements associative subgroupoid.
\end{corollary}

Moreover, these results induce another characterization of a travel groupoid on a complete multipartite graph from the viewpoint of subgroupoid.

\begin{theorem}
Let $(V,*)$ be a travel groupoid. 
Then the following statements are equivalent:
\begin{enumerate}
\item $(V,*)$ is on a complete multipartite graph;
\item $(V,*)$ satisfies the following condition;
\\
every maximal associative subgroupoid $W$ of $V$ satisfies that, for all $v\in V\setminus W$, there is $w\in W$ such that 
\[
 (W-\{ w\})\cup \{ v \}
\]
is a maximal associative subgroupoid of $V$.

\end{enumerate}
\end{theorem}
\begin{proof}
Let $(V,*)$ be a travel groupoid on a complete multipartite graph.
Then a maximal clique corresponds to a maximal associative subgroupoid and obviously the above condition holds.

Conversely, fix a maximal associative subgroupoid $W$.
Let $v\in V\setminus W$ and suppose that there are $w_1, w_2\in W$ such that 
\[
 W_1=(W-\{ w_1 \} )\cup \{ v \} ,
 W_2=(W-\{ w_2 \} )\cup \{ v \}
\]
are maximal associative subgroupoids.
Then, by the associativity of $W_1, W_2$ and $W$, 
\[
x*y=y
\]
holds for all $x,y\in W\cup \{ v \}$.
Thus, by Theorem \ref{asso-tra}, $W\cup \{ v \}$ is an associative subgroupoid of $V$.
This contradicts to the maximality of $W$.
Therefore, for each $v\in V\setminus W$, there is only one $w\in W$ satisfying the above condition.
By using this correspondence, we can define a map $\phi_W :V\setminus W \to W$.

For the map $\phi_W$, a set $\phi_{W}^{-1}(w)\cup \{w \}$ is an independent set of $G_V$.
Since, if there are adjacent vertices $u,v$ in $\phi_{W}^{-1}(w)\cup \{w \}$,
then the set $(W-\{ w \})\cup \{u,v \}$ make an associative subgroupoid. 
This contradicts to the maximality. 

Therefore, we obtain the following partition of $V$
\[
V=\bigcup_{w\in W}\phi^{-1}_{W}(w)\cup\{w \}.
\]

Lastly, for any $w_1, w_2\in W (w_1\neq w_2)$,
the associativity of $(W-\{ w_1 \})\cup \{ u\}$ and $(W-\{ w_2 \})\cup \{ v\}$ implies 
the existence of an edge between  any $u\in \phi_{W}^{-1}(w_1)\cup\{ w_1\}$ and $v\in \phi^{-1}_{W}(w_2)\cup\{ w_2\} $. 


Thus, these imply that the associated graph $G_V$ is a complete multipartite graph.
This completes the proof.
\end{proof}


Next, we characterize a travel groupoid on a complete bipartite graph and on a star graph.
Note that, a complete bipartite graph is a triangle-free complete multipartite graph.

\begin{theorem}
Let $(V,*)$ be a travel groupoid and $|V|>1$.
Then the following statements are equivalent:
\begin{enumerate}
\item $(V,*)$ is on a complete bipartite graph;
\item $(V,*)$ satisfies the following condition;\\
 {\rm (tcb)} if $v*w=w$, then  either $v*u=u$ or $w*u=u$ (for all pairwise distinct  $u,v,w\in V$).
\end{enumerate}
\end{theorem}
\begin{proof}
Let $(V,*)$ be a travel groupoid on a complete bipartite graph $G$. 
Since a complete bipartite graph is a special class of a complete multipartite graph, $(V,*)$ satisfies the condition (tcm).
If $G=K_{1,1}$, $V$ has no three pairwise distinct vertices and $(V,*)$ satisfies (tcb) obviously.
Suppose that $V$ has more than or equal to three vertices, and $u,v,w\in V$ are pairwise distinct vertices satisfying $v*w=w$.
If $v*u=u$ and $w*u=u$, then $u,v,w$ make an associative subgroupoid of $V$.
Thus, by Corollary \ref{asso-clique}, $G$ include a triangle.
This is a contradiction. 

Conversely, suppose that $(V,*)$ satisfies the condition (tcb).
Then, clearly $(V,*)$ satisfies the condition (tcm), and $(V,*)$ is a travel groupoid on a complete multipartite graph. 
If $V$ consists of two vertices, the associated graph $G_V$ is a complete bipartite graph $K_{1,1}$.
Therefore $(V,*)$ is a travel groupoid on a complete bipartite graph.
If $V$ has more than or equal to three vertices and the associated graph $G_V$ has a triangle $\{ u,v,w \}$, then $u,v,w$ satisfies
\[
 v*w=w, v*u=u, w*u=u.
\]
This is a contradiction.
Thus, $G_V$ is a triangle-free complete multipartite graph.
\end{proof}

\begin{definition}
For a groupoid $(V,*)$, a left unit is an element $e\in V$ satisfying
\[
 e*v=v
\]
for all $v\in V$.
\end{definition}

\begin{theorem}
Let $(V,*)$ be a travel groupoid and $|V|>1$.
Then the following statements are equivalent:
\begin{enumerate}
\item $(V,*)$ is on a star graph;
\item $(V,*)$ satisfies (tcb) and has a left unit $e\in V$. 
\end{enumerate}
\end{theorem}
\begin{proof}
It follows from the fact that a star graph is a special class of a complete bipartite graph, and a left unit is an adjacent vertex of all other vertices.
\end{proof}

\section{Number of travel groupoids on complete multipartite graphs}

In the last of this article, 
we give the number of the travel groupoids on the complete multipartite graph.
We denote the multinomial coefficients by
\[
\binom{n}{a_1,\cdots ,a_l}= \frac{n!}{a_{1}!\cdots a_{l}!}.
\]

Let $G=K_{n_1,\cdots ,n_l} $ be a complete multipartite graph with $l$ parts. Suppose that 
\[
V(G)=V=V_1 \cup \cdots \cup V_l,
\]
and let $V_i$ be 
\[
V_i
 =\{ v_{i,1}, \cdots ,v_{i,n_i} \} 
\]
for each $1\leq i\leq l$.
Since a travel groupoid on $K_{n_1,\dots ,n_l}$ is non-confusing, 
the number of the travel groupoids on $K_{n_1,\cdots ,n_l}$ is equal to 
\[
 \left|  \prod_{x\in V}S_{G}(x) \right|
 =\prod_{x\in V } \left| S_{G }(x) \right|
 =\prod_{p=1}^{l}  \prod_{x\in V_p} \left| S_{G}(x) \right| 
\]

\begin{lemma}\label{tree-number}
For all $x\in V_p$, 
\[
\left| S_G(x) \right|
=\sum_{S_p=n_p -1}\binom{ n_{p}-1 }{ m_{1,1},\cdots ,m_{p-1,n_{p-1}} ,m_{p+1,1},\cdots ,m_{l,n_{l} } }
\]
Here, $S_p$ means the following sum
\[
S_p=m_{1,1}+\cdots +m_{p-1,n_{p-1} }+m_{p+1,1}+\cdots +m_{l,n_{l} }.
\]

\end{lemma}
\begin{proof}
Let $v_{p,i}\in V_p$.
Since a $v_{p,i}$-tree is a rooted spanning tree containing all incident edges of the root $v_{p,i}$, 
$v_{p,i}$-tree contains the following edges:
\[
 v_{p,i}v_{1,1}, \cdots ,v_{p,i}v_{p-1,n_{p-1} }, v_{p,i}v_{p+1,1},\cdots v_{p,i}v_{l,n_{l}} .
\]
For the other edges, the vertex $v_{1,1}$ adjacent with $m_{1,1} (0\leq m_{1,1} \leq n_{p}-1)$ vertices except $v_{p,i}$, and its number is equal to
\[
\binom{n_{p}-1}{m_{1,1} }
\]
Similarly, the vertex $v_{1,2}$ adjacent with $m_{1,2} (0\leq m_{1,2} \leq n_{p}-1-m_{1,1})$ vertices except $v_{p,i}$ and the vertices that selected previously. The number is equal to
\[
\binom{n_{p}-1-m_{1,1} }{m_{1,2} } .
\]
Thus, by repeating, we obtain the following results
\begin{eqnarray*}
\left| S_G(v_{p,i}) \right|
&=&
\sum_{ S_p=n_p -1 } \binom{ n_{p}-1 }{ m_{1,1} } \binom{n_{p}-1-m_{1,1} }{ m_{1,2} }\cdots 
\binom{ n_{p}-1-m_{1,1}-\cdots -m_{l,n_{l}-1} }{ m_{l,n_l } } \\
&=&
\sum_{S_p=n_p -1}\binom{ n_{p}-1 }{ m_{1,1},\cdots ,m_{p-1,n_{p-1}} ,m_{p+1,1},\cdots ,m_{l,n_{l} } }
\end{eqnarray*}
This is similar for other $x\in V_p$.
\end{proof}

\begin{theorem}
The number of travel groupoids on the complete multipartite graph $G=K_{n_{1},\cdots ,n_{l} }$ is equal to
\begin{equation}
\prod_{p=1}^{l} \left( \sum_{S_p=n_p -1}\binom{ n_{p}-1 }{ m_{1,1},\cdots ,m_{p-1,n_{p-1}} ,m_{p+1,1},\cdots ,m_{l,n_{l} } } \right)^{n_p}
\end{equation}
Here, $S_p$ means the following sum
\[
S_p=m_{1,1}+\cdots +m_{p-1,n_{p-1} }+m_{p+1,1}+\cdots +m_{l,n_{l} }.
\]
\end{theorem}

\begin{proof}
Since $V_p$ has $n_p$ elements, it follows from the Lemma \ref{tree-number}: 
\begin{eqnarray*}
\left|  \prod_{x\in V} S_{G}(x) \right|
&=& \prod_{p=1}^{l} \prod_{x\in V_p} \left| S_{G }(x) \right|  \\
&=&
\prod_{p=1}^{l} \left( \sum_{S_p=n_p -1}\binom{ n_{p}-1 }{ m_{1,1},\cdots ,m_{p-1,n_{p-1}} ,m_{p+1,1},\cdots ,m_{l,n_{l} } } \right)^{n_p}
\end{eqnarray*}
\end{proof}

\begin{theorem}
The number of simple travel groupoids on the complete multipartite graph $K_{n_1,\cdots ,n_l }$ is equal to 
\begin{equation}\label{simple-cmp}
\prod_{p=1}^{l}\prod_{k=1}^{n_p}\sum_{S_p=n_p-k} \binom{n_{p}-k}{ m_{1,1},\cdots ,m_{p-1,n_{p-1}} ,m_{p+1,1},\cdots ,m_{l,n_{l} } }
\end{equation}
Here, $S_p$ means the following sum
\[
S_p=m_{1,1}+\cdots +m_{p-1,n_{p-1} }+m_{p+1,1}+\cdots +m_{l,n_{l} }.
\]
\end{theorem}
\begin{proof}
By Theorem \ref{simple-tree}, we see that the number of simple travel groupoids on $G$ is equal to the number of the following set
\[
 \mathfrak{S}_{G }=
 \left\{ \{ T_v \}\in \prod_{v\in V}S_{G }(v) |  \exists u-v\mbox{ path on } T_u\cap T_v \ (\forall u,v\in V) \right\}.
\]
Since a $u$-tree and a $v$-tree have a common edge $uv$ for $u\in V_i ,v\in V_j(i\neq j)$. 
The set $\mathfrak{S}_{G }$ is equal to 
\begin{eqnarray*}
  \bigcap_{i=1}^{l} \left\{ \{ T_v \}\in \prod_{v\in V}S_{G }(v) |  \exists v_{i,p}-v_{i,q}\mbox{ path on } T_{v_{i,p} }\cap T_{v_{i,q} } \ (\forall v_{i,p},v_{i,q}\in V_i) \right\} 
\end{eqnarray*}
Thus, it is sufficient to count the number of the $v$-trees $\{ T_v\} $ having $v_{i,p}$-$v_{i,q}$ path on $T_{v_{i,p} }\cap T_{v_{i,q} }$ for all $v_{i,p},v_{i,q} \in V_i, 1\leq i\leq l$.
First, the number of $v_{1,1}$-trees is equal to
\begin{eqnarray*}
 \sum_{S_{1}=n_{1} -1} \binom{n_{1}-1 }{m_{2,1} } \binom{n_{1}-1-m_{2,1} }{ m_{2,2} }\cdots \binom{n_{1}-1-m_{2,1}-\cdots -m_{l,n_{l}-1} }{ m_{l,n_l} } 
= \sum_{S_{1}=n_{1} -1 }\binom{ n_{1} -1 }{ m_{2,1}, \cdots ,m_{l,n_l } } 
\end{eqnarray*}
also the number of $v_{1,2}$-trees having a common path between $v_{1,1}$ and $v_{1,2}$ is equal to
\begin{eqnarray*}
\sum_{S_{1}=n_{1} -2} \binom{n_{1}-2 }{m_{2,1} } \binom{n_{1}-2-m_{2,1} }{ m_{2,2} }\cdots \binom{n_{1}-2-m_{2,1}-\cdots -m_{l,n_{l}-1} }{ m_{l,n_l} } 
=\sum_{S_{1}=n_{1} -2 }\binom{ n_{1} -2 }{ m_{2,1}, \cdots ,m_{l,n_l } } 
\end{eqnarray*}
By repeating this counting for all vertices, we obtain the relation (\ref{simple-cmp}).
\end{proof}


\bibliographystyle{line}
\bibliography{JAMS-paper}


\end{document}